\documentclass{amsart}
\usepackage{graphicx}
\usepackage{amscd}
\usepackage{amsmath}
\usepackage{amsfonts}
\usepackage{amssymb}
\theoremstyle{plain}
\newtheorem{theorem}{Theorem}%[section]

\newtheorem{lemma}[theorem]{Lemma}
\newtheorem{proposition}[theorem]{Proposition}
\theoremstyle{definition}
\newtheorem{example}[theorem]{Example}

\newtheorem{definition}[theorem]{Definition}

\newtheorem{remark}[theorem]{Remark}
\theoremstyle{remark}

\begin{document}
\title{Divisorial Multiplicative Lattices}

\author{Tiberiu Dumitrescu and Mihai Epure}

\address{Facultatea de Matematica si Informatica,University of Bucharest,14 A\-ca\-de\-mi\-ei Str., Bucharest, RO 010014,Romania}
\email{tiberiu\_dumitrescu2003@yahoo.com}

\address{Simion Stoilow Institute of Mathematics of the Romanian Academy, Research unit 6, P. O. Box 1-764, RO-014700 Bucharest, Romania}
\email{epuremihai@yahoo.com, mihai.epure@imar.ro}


\begin{abstract}\noindent
We prove several fundamental results about divisorial integral domains in the setup of multiplicative lattices.
\end{abstract}

\maketitle
%
%
%\section{Divisorial closure in lattices}

In \cite{He} Heinzer initiated the study of integral domains in which all nonzero ideals are divisorial (later called {\em divisorial domains}) establishing their   fundamental properties. 
He showed \cite[Theorems 2.4 and 2.5]{He} that a divisorial domain $D$ is {\em h-local}, that is, 
$D/P$ (resp. $D/(x)$) is quasilocal (resp. semi-quasilocal) for each nonzero prime ideal $P$ (resp. for each $x\in D-\{0\}$). He also showed 
\cite[Theorem 5.1]{He} that an integrally closed domain $D$ is divisorial iff $D$ is an h-local Pr\"ufer domain whose maximal ideals are finitely generated. He also showed \cite[Proposition 5.5]{He} that a completely integrally closed domain $D$ is divisorial iff $D$ is a Dedekind domain.

The aim of this short note is to rewrite  most of the results in \cite{He}  in the language of Abstract Ideal Theory in the sense of \cite{D} and \cite{A}.
We recall some standard terminology and facts.
 A {\em multiplicative lattice} is a complete lattice  $(L,\leq)$ (with bottom element $0$ and top element $1$) which is also  a commutative monoid with identity $1$ (the top element) such that
 $$a( \bigvee_\alpha b_\alpha) = \bigvee_\alpha (ab_\alpha)  \mbox{ \ for each } a,b_\alpha\in L.$$
When $x,y\in L$ and $x\leq y$, we say that $x$ is {\em below} $y$ or that $y$ is {\em above} $x$.
An element $c\in L$ is {\em compact} if 
 \begin{center}
$S\subseteq L$ and $c\leq \bigvee S$  implies $c\leq \bigvee T$ for some finite subset $T\subseteq S.$
\end{center}
\noindent
An element  in $x\in L$ is {\em proper} if $x\neq 1$.
When $1$ is compact, every proper element is below some {\em maximal} element (i.e. maximal in $L-\{1\}$). 
Let $Max(L)$ denote the set of maximal elements of $L$. 
%By ``$(L,m)$ is  local'', we mean that $Max(L)=\{m\}$.
An element $p\in L-\{1\}$ is {\em prime} if  $xy \leq p$ (with $x,y \in L$) implies $x \leq p$ or $y\leq p$.   Every maximal element is prime. $L$ is a {\em lattice domain} if $0$ is a prime element.
For  $x,y\in L$,   $(y : x)$ denotes the element $\bigvee \{a \in L\ | \ ax \leq y\}$.
An element $x\in L$ is {\em principal} if it is both 
{\em meet-principal}, i.e. satisfies
$$y \wedge zx = ((y : x) \wedge z) x  \ \ \  \mbox{ for all }  y, z \in L $$
and 
{\em join-principal}, i.e. satisfies
$$y \vee (z : x) = ((yx \vee z) : x) \ \ \mbox{ for all }  y, z \in L .$$
 If $x$ and $y$ are principal elements, then so is $xy$. Conversely, if  $L$ is a lattice domain and $xy$ is a nonzero principal element, then $x$ and $y$ are principal. 
In a lattice domain, every nonzero principal element $x$ is cancellative, that is $xy = xz$ ($y,z\in L$) implies $y = z$. 

The lattice  $L$ is {\em principally generated} if every element is a join of principal elements.
$L$ is a {\em $C$-lattice} if $1$ is compact, the set of compact elements  is closed under  multiplication  and  every element is a join of compact elements.
  In a $C$-lattice every principal element is compact.
 
 The $C$-lattices have a well behaved localization theory (see \cite{JJ}).  Let $L$ be a $C$-lattice and $L^*$ the subset of its compact elements. For $p \in L$ a prime element and $x\in L$, 
 the {\em localization} of $x$ at $p$ is
 $$x_p = \bigvee \{a \in L^* \ |\ as \leq x \mbox{ for some } s \in L^* \mbox{ with } s\not\leq p\}.$$  
 Then $L_p:=\{ x_p;\ x\in L\}$ is  again a lattice with multiplication $(x,y)\mapsto (xy)_p$, join $\{ (b_\alpha)\}\mapsto 
 (\bigvee b_\alpha)_p $ and meet $\{ (b_\alpha) \}\mapsto  (\bigwedge b_\alpha)_p$.
For $x,y\in L$, we have:
$(i)$ $x\leq x_p$, $(x_p)_p=x_p$, 
 $(x \wedge y)_p = x_p \wedge y_p $,
 and $x_p = 1$ iff $x\not\leq p$.
$(ii)$  $x=y$ iff $x_m=y_m$ for each $m\in Max(L)$.
$(iii)$ $(y : x)_p \leq (y_p : x_p)$ with equality if $x$ is compact.
$(iv)$ The set of compact elements of $L_p$ is $\{c_p\ |\ c\in L^*\}$.
% 
%$(v)$ A compact element $x$ is principal iff $x_m$ is principal for each $m\in Max(L)$. For basic results or teminology our references are \cite{G}, \cite{D} and \cite{A}.
\begin{center}
{\em From now on, by term {\em lattice} we mean a principally generated C-lattice.}
\end{center}
Our first task is to
 define an "inside" divisorial closure  on a lattice  domain (for a more natural "outside" approach see  for instance \cite{F} and \cite{M}).

\begin{lemma} \label{2} 
Let $L$ be a lattice domain and  $a,x,y,z\in L-\{0\}$ such that $x,y,z$ are principal elements and $x,y\leq a$. Then 

$(i)$ $z(y:a) = (zy:a)$,

$(ii)$ $(x:(x:a)) = (y:(y:a)).$
\end{lemma} 
\begin{proof} 
$(i)$ Since $(zy:a) \leq (zy:y) = z$ and $z$ is principal, we have 
$$(zy:a) = z((zy:a):z) = z(zy:az) = z(y:a).$$
$(ii)$ By   $(i)$ we have
$$(xy:(xy:a)) = (xy:y(x:a)) = (x:(x:a)).$$ As  $x,y$ play symmetric roles, we are done.
\end{proof}

Lemma \ref{2} justifies the following key definition.

\begin{definition} %\label{} 
Let $L$ be a lattice domain. For $a\in L-\{0\}$, we  define its {\em divisorial closure} (or {\em v-closure}) by 
$$a_v := (x:(x:a)) \mbox{ for some principal element }  x \leq a. $$
Say that $a$ is divisorial if $a=a_v$. For completness, we postulate that $0$ is divisorial.
\end{definition}

Note that our definition agrees with the well-known divisorial closure on an integral domain. We list some basic properties.

\begin{proposition} \label{81} 
Let $L$ be a lattice domain and  $a,b\in L-\{0\}$. We have

$(i)$  $a \leq a_v$ (so $1$ is divisorial).

$(ii)$ If $x\leq a$ is a nonzero principal element, then $(x:a)$ is a divisorial element (so all nonzero principal elements are divisorial).

$(iii)$ If $a \leq b$, then $a_v \leq b_v$. 

$(iv)$ $(xa)_v = xa_v$ for each nonzero principal element  $x$.

$(v)$ $(a_v)_v = a_v$.

$(vi)$ If $(a_i)_{i\in I}$ is a family of divisorial elements of $L$, then 
$\bigwedge_{i\in I} a_i$ is divisorial.
\end{proposition} 
\begin{proof} 
$(i)$ Let  $x\leq a$ be a nonzero principal element. As $a(x:a) \leq x$, we get $a\leq (x:(x:a)) = a_v$.
$(ii)$ We have $(x:a)_v=(x:(x:(x:a))) = (x:a)$. 
\\ $(iii)$ Let $x\leq a$ be a nonzero principal element. Since $a \leq b$, we have $(x:a)\geq (x:b)$, so $a_v = (x:(x:a))\leq  (x:(x:b)) = b_v$.
$(iv)$ Let $x,y$ be   nonzero principal elements with $y\leq a$.
Then   $(xa)_v = (xy:(xy:xa)) = (xy:(y:a)) = x(y:(y:a)) =xa_v$, where the third equality comes from Lemma \ref{2}.
$(v)$ follows from $(ii)$
and
$(vi)$ follows from  $(\bigwedge_{i\in I} a_i)_v \leq \bigwedge_{i\in I} {(a_i)}_v=\bigwedge_{i\in I} a_i.$ 
\end{proof}

\begin{example}  \label{15}
If $L$ is a lattice domain and  $0 < m^2 \leq x < m$ where $m$ is a  maximal element and $x$   a principal element, then $m=(x:m)$ is divisorial. As a specific example, $M=(2,1+\sqrt{-3})$ is a divisorial ideal of $\mathbb{Z}[\sqrt{-3}]$ since $M^2\subseteq (2)$.
\end{example}

\begin{example} 
Consider  the semiring $\mathbb{N}$ of non-negative integers. 
It is well known that the nonprincipal ideals of $\mathbb{N}$ are scalar multiples of numerical semigroups. One can see that   the divisorial closure of a numerical semigroup is the whole $N$. 
Thus in the   ideal lattice of $\mathbb{N}$ the divisorial elements are only the principal elements.
\end{example}

\begin{proposition} \label{12} 
Let $L$ be a lattice domain and $a\in L-\{0\}$. Then  
$$a_v = \bigwedge \{ (x:y)\ |\ x,y \mbox{ are principal elements and\  }  a\leq (x:y)       \}.
$$
\end{proposition} 
\begin{proof} 
The right hand side $d$  of  equality above is divisorial (cf. Proposition \ref{81}) and $a\leq d$, hence   $a_v\leq d$.
Conversely, let $x,y$ be   nonzero principal elements with $y\leq d$, 
$x\leq a$. For each nonzero principal element $z\leq (x:a)$  we get $a\leq (x:z)$, so $y\leq d\leq (x:z)$, hence $yz\leq x$.  It follows that $y\leq (x:(x:a))=a_v$, thus $d\leq a_v$.
\end{proof}

\begin{definition} %\label{} 
Say that a lattice domain $L$ is {\em divisorial} if its  elements  are divisorial. So an integral domain $D$ is divisorial iff  so is the ideal lattice of $D$.
\end{definition}

Clearly a {\em Dedekind lattice} (i.e. a lattice domain whose elements are principal (see \cite{AJa})) is a divisorial lattice.

\begin{example}  \label{17}
Consider the ideal lattice of the local domain 
$\mathbb{Z}[\sqrt{-3}]_{(2,1+\sqrt{-3})}$. Set $a=(2)$, $b=(1+\sqrt{-3})$, 
$c=(1-\sqrt{-3})$,  $m=(2,1+\sqrt{-3})$. Note that
$a^2=bc$, $b^2=ac$, $c^2=ab$, $m^2=am=bm=cm$. In fact $L$ consists of 
$$1 > m > a,b,c > am >  a^2,ab,ac > a^2m >  a^3,a^2b,a^2c > a^3m > ...\ 
$$
where the nonprincipal elements are $a^nm$, $n\geq 0$. As $m=(a:m)$ is divisorial, $L$ is divisorial.
See \cite{PZ} for more examples of this type.
\end{example}

\begin{remark}  \label{16}
Let $L$ be a lattice domain and let $H$ be the cancellative monoid consisting of the principal elements of $L$. 
We observe the following link between the $v$-closures in $L$ and $H$.
Denote by $M_{v'}$  the $v$-closure of some 
$M\subseteq H$ (see \cite[(11.4)]{Ha}).
Since Proposition \ref{12} can be restated as
$$a_v = \bigvee \{ x\in H\ |\ (y,z\in H \mbox{ and \ }  za\leq y) \mbox{ implies  }   zx\leq y     \}
$$
we see that
$$ a_v = \bigvee \{ x\in H\ |x\leq a\}_{v'}.
$$
Using this fact  one can see that if $L$ is divisorial, then $L$   is isomorphic to the lattice of $v$-ideals of $H$. For instance, the ideal lattice in Example \ref{17} is isomorphic to the lattice of $v$-ideals of
the monoid $$ < a,b,c\ |\  a^2=bc,\ b^2=ac,\ c^2=ab>.
$$
\end{remark}

\begin{remark} \label{3} 
Let $L$ be a divisorial lattice domain.
Clearly, if $a,b,x\in L$ and   $x$ is a principal element $\leq a \wedge b$,  
then $a \leq b$ iff $(x:a) \geq (x:b)$, so
$a = b$ iff $(x:a) = (x:b)$
\end{remark}

Next we  extend to lattices the  results (2.1), (2.3), (2.4), (2.5), (3.6), (5.1) and (5.5) from \cite{He}.

\begin{lemma} \label{4} 
Let $L$ be a divisorial lattice domain.
Suppose  $(a_i)_{i\in I}$ is a family of  elements of $L$ and $x\leq \bigwedge_{i\in I} a_i$ is a nonzero principal element. Then  
$$(x:\bigwedge_{i\in I} a_i) = \bigvee_{i\in I} (x:a_i).$$
\end{lemma} 
\begin{proof} 
We have $$ \bigvee_{i\in I} (x:a_i)=(x:(x : \bigvee_{i\in I} (x:a_i))) = (x:\bigwedge_{i\in I} (x:(x:a_i))) = (x:\bigwedge_{i\in I} a_i ).$$
\end{proof}

\begin{theorem} \label{5}
Let $L$ be a divisorial lattice domain, $a\in L-\{0,1\}$ and $p$ some maximal element $\geq a$. Then
 $$\bigwedge  \{b\in L\ |\ a\leq b \not\leq p\ \}\not\leq p$$
that is, there exists a smallest element, denoted $a(p)$, satisfying $a\leq a(p) \not\leq p.$ Moreover, if $a$ is a prime element, then $a(p)=1$.
\end{theorem} 
\begin{proof} 
Set $B=\{b\in L\ |\ a\leq b \not\leq p\ \}.$ We have to show that $\bigwedge  B  \not\leq p$.
Note that $1\in B$, so $B$ is nonempty.
Fix  a  nonzero principal element $x\leq a$.
By Remark \ref{3} we have  $x =(x:1) < (x:p)$, so we may pick 
a principal element $y\leq  (x:p)$ with $y\not\leq x$.
Let $C$ be a finite subset of $B$ (so clearly $\bigwedge C \not\leq p$, that is, $p \vee \bigwedge C=1$).
If $y \leq (x:\bigwedge C)$, then $y=y\cdot 1 = y(p \vee \bigwedge C) \leq x$,  a contradiction. It remains that 
$y \not\leq  (x:\bigwedge C) = \bigvee\{(x:c)\ |\ c \in C\}$, cf. Lemma \ref{4}. As $y$ is compact, we get 
 $y \not\leq  \bigvee \{(x:b);\ b \in B\} = (x:\bigwedge B)$, cf. Lemma \ref{4}. So $(x:p) \not\leq    (x:\bigwedge B)$, thus  $\bigwedge B \not\leq  p$, cf. Remark \ref{3}.
 
Suppose now that $a$ is prime and let $y\leq a(p)$ be a principal element with $y\not\leq p$. As $a\leq a\vee y^2\not\leq p$, the first part  shows that $y\leq a(p)\leq a\vee y^2$. As $y$ is principal, $a$ is prime  and $y\not\leq a$, we get $$1=(a\vee y^2):y =(a:y) \vee y = a\vee y\leq a(p)$$
so $a(p)=1$. 
\end{proof}

We continue to use notation   $a(p)$ established in Theorem \ref{5}. 

\begin{definition}
Say that a lattice domain $L$ is {\em h-local} if 

$(i)$ every  nonzero prime element of $L$ is below a   unique maximal element and 

$(ii)$ for every  $b\in L-\{0\}$, the set of maximal elements above $b$ is finite.
\\
So an integral domain $D$ is h-local iff  its ideal lattice is so.
\end{definition}

% ramane label 6

\begin{theorem} \label{11}
A divisorial lattice domain $L$ is h-local.
\end{theorem}
\begin{proof} 
$(i)$ Suppose  $0 < r\leq p \wedge m$ where $r$ is prime and 
$p$, $m$ are   distinct maximal elements. By Theorem \ref{5}, we get the   contradiction $1=r(p)\leq m$. 

$(ii)$ Let $a\in L-\{0,1\}$ and 
  $\Gamma$   the set of maximal elements above $a$. As $a(m)\not\leq m$ for each $m\in  \Gamma$, we get $\bigvee_{m\in  \Gamma} a(m) =1$, so $\bigvee_{m\in  \Lambda} a(m) =1$ for some finite subset $\Lambda$ of $\Gamma$, because $1$ is compact. Then $\Gamma=\Lambda$ is finite. Indeed, 
if $n\in \Gamma-\Lambda$, then  $n\geq \bigvee_{m\in  \Lambda} a(m) =1$, a contradiction.
\end{proof}

\begin{theorem}  \label{8} 
A lattice  domain  $L$ is divisorial iff it is h-local and $L_m$ is divisorial for each maximal element $m$. 
\end{theorem} 
\begin{proof} 
$(\Rightarrow)$ $L$ is h-local by Theorem \ref{11}. Let $m\in Max(L)$ and $a\in L_m-\{0\}$. Pick a nonzero principal element $x\leq a$. By 
\cite[Proposition 19]{DE}, we have 
$$(x_m:(x_m:a_m)) = (x:(x:a))_m = a_m =a.$$ 
As $x_m$ is principal in $L_m$, $a$ is divisorial in $L_m$.

$(\Leftarrow)$ Let $a\in L-\{0\}$ and $x\leq a$ a nonzero principal element.
Let $m\in Max(L)$. As $L_m$ is divisorial and $L$ is h-local, \cite[Proposition 19]{DE} shows that
$$(a_v)_m= (x:(x:a))_m = (x_m:(x_m:a_m)) = a_m.$$ Thus  we have checked locally that $a_v = a$.
\end{proof}

\begin{definition} %\label{} 
Let $L$ be a lattice domain. Say that  $L$ is a {\em Pr\"ufer lattice} if all compact elements of $L$ are principal. Say that $L$ is {\em completely integrally closed} (resp. {\em integrally closed}) if  $(xc:c)=x$ for 
each nonzero principal element $x$ and each nonzero element (resp. each nonzero compact element) $c$. 

So an integral domain $D$ is  Pr\"ufer, completely integrally closed or integrally closed iff so is its ideal lattice respectively (see \cite[Theorems 34.3 and 34.7]{G}).
\end{definition}

It follows at once that a Pr\"ufer lattice is integrally closed and a Dedekind lattice is completely integrally closed.

\begin{lemma}  \label{9} 
Let $L$ be a divisorial lattice domain and $c\in L-\{0\}$  such that \begin{center} $(cz:c)=z$ for each nonzero principal element $z$.\end{center} Then $c$ is principal.
\end{lemma} 
\begin{proof} 
Let $x,y$ be nonzero principal elements such that $x\leq c$ and $y\leq (x:c)$. Using  the hypothesis and Lemma \ref{2} we get 
$$ (xy:c(x:c)) = ((xy:(x:c)):c)= (y(x:(x:c)):c) = (yc:c) = y.
$$
Hence
$$ c(x:c)  = (xy:(xy:c(x:c))) = (xy : y) = x
$$
thus $c$ is principal, cf. \cite[Theorem 7]{AJ}.
\end{proof}

From Lemma \ref{9}, we get at once:

\begin{theorem}  
A   completely integrally  closed lattice domain is divisorial iff it is Dedekind.
\end{theorem} 

Recall that a linearly ordered lattice domain is called a 
{\em valuation lattice}. It is well-known that a lattice domain $L$ is Pr\"ufer iff $L_m$ is a valuation lattice for each maximal element $m$.

\begin{lemma}  \label{14}
A valuation lattice $L$ is divisorial iff its maximal element $m$ is principal.
\end{lemma} 
\begin{proof} 
$(\Rightarrow)$ 
By Proposition \ref{12}, $m=(x:y)$ for some principal elements $x,y$. So $m$ is principal.
$(\Leftarrow)$ Suppose that $a$ is a nondivisorial element and pick a principal element $a < x \leq a_v$. 
We get  $x=a_v=(x(a:x))_v = x(a:x)_v
$, hence $1=(a:x)_v\leq m_v = m$, a contradiction.

\end{proof}

\begin{theorem}  \label{10}
An integrally closed lattice domain $L$ is divisorial iff
$L$ is   Pr\"ufer, h-local   and its maximal elements are principal.
\end{theorem} 
\begin{proof} 
$(i)\Rightarrow (ii)$.
By Theorem \ref{11}, $L$ is  h-local. Let $c\in L-\{0\}$ be a compact element. As $L$ is integrally closed and divisorial, Lemma \ref{9} shows that $c$ is principal. Hence $L$ is a  Pr\"ufer lattice. 
Let $m$ be a maximal element. By Proposition \ref{12}, $m=(x:y)$ for some principal elements $x,y$. So $m$ is principal. Indeed $x\vee y$ is a principal element $z$, so $x=zx'$ and $y=zy'$ for some principal elements $x',y'$,   cf. \cite[Theorem 7]{AJ}. So 
$z = zx'\vee zy'=z(x'\vee y')$, hence $x'\vee y'=1$, thus $(x:y)=(x':y')=x'$.
$(ii)\Rightarrow (i)$ follows from Theorem \ref{8} and Lemma \ref{14}.
\end{proof}

\begin{remark}
 Let $L$ be a Pr\"ufer lattice. Then $L$ is modular because it is locally totally ordered. By \cite[Theorem 3.4]{A}, $L$ is isomorphic to the ideal lattice  of some Pr\"ufer integral domain. It follows that a divisorial  integrally closed lattice domain  is isomorphic to the ideal lattice  of some divisorial  Pr\"ufer domain.
\end{remark}

\end{document}